\documentclass{amsart}

\usepackage[T1]{fontenc}
\usepackage[utf8x]{inputenc}
\usepackage{lmodern, euscript}
\usepackage{enumerate, graphics}
\usepackage{mathabx}
 \usepackage[all,cmtip]{xy}

\usepackage{amsfonts,amsmath,amstext,amsbsy,amssymb,amsbsy,
  amsopn,amsthm,amscd} 
\usepackage[T1]{fontenc}
\usepackage{hyperref}
\usepackage{mathrsfs}

\RequirePackage[dvipsnames]{xcolor} 
\definecolor{halfgray}{gray}{0.55} 
\definecolor{webgreen}{rgb}{0,0.5,0}
\definecolor{webbrown}{rgb}{.6,0,0} \hypersetup{%
  colorlinks=true, linktocpage=true, pdfstartpage=3,
  pdfstartview=FitV,%
  breaklinks=true, pdfpagemode=UseNone, pageanchor=true,
  pdfpagemode=UseOutlines,%
  plainpages=false, bookmarksnumbered, bookmarksopen=true,
  bookmarksopenlevel=1,%
  hypertexnames=true,
  pdfhighlight=/O,
  urlcolor=webbrown, linkcolor=RoyalBlue,
  citecolor=webgreen, 
  pdftitle={Recap},%
  pdfsubject={},%
  pdfkeywords={},%
  pdfcreator={pdfLaTeX},%
  pdfproducer={LaTeX with hyperref}%
}

\newtheorem{theorem}{Theorem}[section]
\newtheorem{add}[theorem]{Addendum}
\newtheorem{lemma}[theorem]{Lemma}
\newtheorem{corollary}[theorem]{Corollary}
\newtheorem{proposition}[theorem]{Proposition}
\theoremstyle{definition}

\newtheorem{remark}[theorem]{Remark}

\newcommand{\field}[1]{\mathbb{#1}}
\newcommand{\R}{\field{R}}

\newcommand{\cF}{\mathcal{F}}

\newcommand{\cH}{\mathcal{H}}

\newcommand{\cU}{\mathcal{U}}

\newcommand{\ie}{{\it i.e., } }
\newcommand{\eg}{{\it e.g., } }

\renewcommand{\phi}{\varphi}

\newcommand{\eps}{\varepsilon}

\renewcommand{\|}{\,\Vert\,}

\begin{document}
\baselineskip=14pt

\title{Smooth rigidity for higher dimensional contact Anosov flows}
\author {Andrey Gogolev and Federico Rodriguez Hertz}\thanks{The authors were partially supported by NSF grants DMS-1955564 and DMS-1900778, respectively}

 \address{Department of Mathematics, The Ohio State University,  Columbus, OH 43210, USA}
\email{gogolyev.1@osu.edu}

\address{Department of Mathematics, The Pennsylvania State University,  University Park, PA 16802, USA}
\email{hertz@math.psu.edu}

\begin{abstract} 
  \begin{sloppypar}
  We apply the matching functions technique in the setting of contact Anosov flows which satisfy a bunching assumption. This allows us to generalize the 3-dimensional rigidity result of Feldman-Ornstein~\cite{FO}. Namely, we show that if two such Anosov flows are $C^0$ conjugate, then they are $C^{r}$ conjugate for some $r\in[1,2)$ or even $C^\infty$ conjugate under some additional assumptions. This, for example, applies to $1/4$-pinched geodesic flows on compact Riemannian manifolds of negative sectional curvature. We can also use our result to recover Hamendst\"adt's marked length spectrum rigidity result for real hyperbolic manifolds.
  \end{sloppypar}
\end{abstract}
\maketitle

\section{Introduction}

Let $M$ be a closed smooth Riemannian manifold. Recall that a smooth flow $\phi^t\colon M\to M$ is called {\it Anosov} if the tangent bundle admits a $D\phi^t$-invariant splitting $TM=E^s\oplus X\oplus E^u$, where $X$ is the generator of $\phi^t$, $E^s$ is uniformly contracting and $E^u$ is uniformly expanding under $D\phi^t$. 

In this paper will always assume that $M$ has an odd dimension $2d+1$ and that $M$ is equipped with a contact form $\alpha$. Recall that a 1-form $\alpha$ is called {\it contact} if $\alpha\wedge (d\alpha)^d$ is a non-vanishing top-dimensional form. We will consider Anosov flows $\phi^t$ which are also contact. This means that $\phi^t$ preserves a contact form $\alpha$: $\alpha(D\phi^tv)=\alpha(v)$ for all $v\in T M$ and all $t\in\R$ or, equivalently, $X\alpha=0$. 

Basic examples of contact Anosov flows are geodesic flows in negative sectional curvature and more sophisticated examples can be constructed, in particular, in dimension 3~\cite{FH}.

 Recall that flows $\phi_1^t$ and $\phi_2^t$ are called {\it conjugate} if there exists a homeomorphism $h$ such that $h\circ \phi_1^t=\phi_2^t\circ h$ for all $t\in \R$. 

 In the setting of 3-dimensional contact Anosov flows, Feldman and Ornstein proved that any topological (merely $C^0$)  conjugacy is, in fact, $C^\infty$ smooth~\cite{FO}. To formulate our generalization recall that the distributions $E^s\oplus X$ and $E^u\oplus X$ are known to integrate to foliations $W^{0s}$ and $W^{0u}$, respectively, which are called weak stable and weak unstable foliations.
 
\begin{theorem}
\label{thm_flows}
Let $\phi_1^t\colon M_1\to M_1$ and $\phi_2^t\colon M_2\to M_2$ be contact Anosov flows, which are conjugate via a homeomorphism $h\colon M_1\to M_2$. Assume that the weak stable and unstable distributions of $\phi_1^t$ and $\phi_2^t$ are $C^{r}$ for some $r\ge 1$. Then $h$ is $C^{r_*}$.
\end{theorem}
Here $r_*=r$ if $r$ is not integer and $r_*=r-1+Lip$ if $r$ is an integer (if $r=1$ we can set $r_*=1$ as well). The latter means that $h$ is $C^{r-1}$ diffeomorphism with Lipschitz $(r-1)$-jet. Note that $M_1$ and $M_2$ are homeomorphic via $h$, but a priori may carry different smooth structures. We then conclude that they are, in fact, diffeomorphic once we know that $h$ is $C^1$.

We also recall the definition of a distribution $E\subset TM$ being $C^r$. This mean that $E$ is $C^r$ when viewed as a map from $M$ into the grassmann  bundle $\textup{Gr}^{\dim E}(M)$. Alternatively, $E$ s locally spanned by $\dim E$ independent $C^r$ vector fields on $M$.

The main setup  where this result applies is when the Anosov flows satisfy a {\it bunching condition}, which guarantees $C^{r}$ regularity of weak distributions. Denote by $m(A)=\|A^{-1}\|^{-1}$ the conorm of a linear operator $A$. If for some $t>0$ and all $x\in M_i$
$$
\|D\phi_i^t|_{E_i^s(x)}\|\cdot \|D\phi_i^t|_{E_i^u(x)}\|^{r}<m(D\phi_i^{t}|_{E_i^u(x)})
$$
then $E^{0s}_i$, the weak stable distribution of $\phi_i^t$ is $C^{1+\eps}$~\cite{Hass}. Similarly, if 
$$
\|D\phi_i^{t}|_{E_i^s(x)}\|<m(D\phi_i^{t}|_{E_i^u(x)})\cdot  m(D\phi_i^{t}|_{E_i^s(x)})^{r}
$$
then the weak unstable distribution $E^{0u}_i$ is also $C^{r}$. In general, these conditions are optimal for $C^{r}$ smoothness of weak distributions~\cite{Hass}.

These bunching conditions can be verified for some specific examples. In particular, a geodesic flow on $1/4$-pinched negatively curved Riemannian manifold satisfies the above conditions with $r=1$ and, hence, has $C^1$ weak stable and unstable distributions. The $a^2$-pinching condition means that the sectional curvature function $K$ is bounded above and below as follows:
$$
-c< K\le- a^2c
$$ 
where $c$ is a positive constant. Hence, Theorem~\ref{thm_flows} applies to geodesic flows on Riemannian manifolds which are $C^2$ close to a hyperbolic manifold. Also point-wise $1/2$-pinching implies that weak distributions are $C^1$~\cite{Hass2}.

Now we present some corollaries of our main result. Note that by taking the product of the above bunching inequalities we can see that they are never simultaneously satisfied if $r\ge 2$. Hence, in practical terms, Theorem~\ref{thm_flows} only yields a limited regularity of the conjugacy: somewhere between $C^1$ and $C^2$. However, we can remedy this under some additional assumptions. We need to introduce another condition which we call {\it conformal $r$-pinching}. An Anosov flow $\phi^t$ satisfies conformal $r$-pinching with $r\in(1,2]$ if for a sufficiently large $t$ and all $x\in M$
$$
 \|D\phi^t|_{E^u(x)}\|< m(D\phi^t|_{E^u(x)})^{r}\,\,\,\,\,\mbox{and}\,\,\,\,\,\,m(D\phi^t|_{E^s(x)})^{r}< \|D\phi^t|_{E^s(x)}\|
 $$

\begin{corollary}
\label{thm_flows2}
Let $\phi_1^t\colon M_1\to M_2$ and $\phi_2^t\colon M_1\to M_2$ be contact Anosov flows, which are conjugate via a homeomorphism $h\colon M_1\to M_2$. Assume that the weak stable and unstable distributions of $\phi_1^t$ and $\phi_2^t$ are $C^{r}$ for some $r> 1$. Also assume that  $\phi_1^t$ and $\phi_2^t$ are conformally $r$-pinched. 
Then $h$ is a $C^\infty$ diffeomorphism.
\end{corollary}

\begin{remark} In the above corollary one can replace the pinching assumption with an assumption about existence of a conformal periodic point. This is a periodic point $p=\phi_1^T(p)$ such that the linearized return map $D\phi_1^T\colon T_pM_1\to T_pM_1$ is conformal on $E^u_1(p)$ and $E^s_1(p)$. This modified statement can be proved with a different bootstrap argument recently used by the authors in~\cite{GRH3}. While more ad hoc, the assumption about existence of conformal periodic point does cover some flows to which the above corollary does not apply.
\end{remark}

\begin{corollary}
\label{cor2b}
Let $\phi_i^t:T^1N_i\to T^1N_i$ be geodesic flows on negatively curved manifolds $(N_i,g_i)$, $i=1,2$ which are $C^0$ conjugate. Assume that both metrics $g_1$ and $g_2$ are $1/2$-pinched. Then the conjugacy is $C^\infty$ smooth.
\end{corollary}
This, in particular, applies to geodesic flows of Riemannian metrics in a sufficiently small $C^2$-neighborhood of a hyperbolic metric: if two such metric have the same marked length spectrum (or, equivalently, are $C^0$ conjugate) then the conjugacy of geodesic flows is a $C^\infty$ diffeomorphism. 
\begin{corollary} 
\label{cor2}
Let $\phi^t$ be a geodesic flow on a negatively curved $1/2$-pinched  manifold. Then there exists a $C^1$-neighborhood $\cU$ of $\phi^t$ such that if $\phi_1^t, \phi_2^t\in \cU$ are contact and conjugate, then the conjugacy is $C^\infty$ smooth.
\end{corollary}

We can also partially recover a geometric rigidity result of Hamendst\"adt~\cite{Ham}.

\begin{corollary}
\label{cor3}
If $M$ and $N$ are closed negatively curved manifolds with the same marked length spectrum and $C^1$ Anosov splittings, then $M$ and $N$ have the same volume.
\end{corollary}
Our result is weaker than the result of Hamendst\"adt~\cite{Ham} because Hamendst\"adt only assumed that the Anosov splitting of  $TT^1M$ is $C^1$ and didn't have any assumption on the Ansov splitting of $TT^1N$. Still it is enough to recover marked length spectrum rigidity of hyperbolic manifolds using the Besson-Courtois-Gallot entropy rigidity theorem~\cite{BCG}. Hence, following Hamendst\"adt's application of entropy rigidity we arrive at a version of marked length spectrum rigidity for hyperbolic manifolds.

\begin{corollary}
Let $(M, g_1)$ be a closed real hyperbolic manifold of dimension $\ge3$ and let $g_2$ be a $1/4$-pinched Riemannian metric on $M$. Assume that $(M, g_1)$ and $(M, g_2)$ have the same marked length spectrum. Then $g_2$ is isometric to $g_1$.
\label{cor4}
\end{corollary}

\subsection{Organization} 
In the next section we recall some facts about contact Anosov flows and about the matching function technique. Then we introduce the main technical tool which we call the Subbundle Theorem. In Section~3 we prove Theorem~\ref{thm_flows} and in Section~4 we derive all the corollaries.

We would like to thank the anonymous referee for a thorough and beautiful report.

\section{Preliminaries}

\subsection{Basic facts about contact Anosov flows}

Recall that we denote by $W^{s}$, $W^{u}$, $W^{0s}$ and $W^{0u}$ the stable, unstable, weak stable and weak unstable foliations of an Anosov flow. When needed, we will also use a subscript $i$ to indicate dependence on the flow $\phi_i^t$, $i=1,2$.

It is immediate from the definition of Anosov contact flow $\phi^t$ that $\alpha(X)$ is constant; hence, we can normalize the contact form so that $\alpha(X)=1$. Also we have $\ker\alpha=E^s\oplus E^u$. Indeed if $v\in E^s$ then $\alpha(v)=\alpha(D\phi^t(v))\to 0$ as $t\to\infty$, and similarly for $v\in E^u$. It is a simple exercise to check that if $\phi^t$ is a contact Anosov flow then $\dim E^s=\dim E^u=d$.

\begin{lemma}
\label{lemma_contact}
Let $\phi^t\colon M\to M$ be a contact Anosov flow with $C^1$  stable and unstable foliations. Assume that the stable foliation $W^s$ admits a $C^1$ subordinate foliation $\cF$, $\cF(x)\subset W^s(x)$, $x\in M$, which integrates jointly with $W^u$. Then $\cF$ is a foliation by points, that is, $\cF(x)=\{x\}$ for all $x\in M$.
\end{lemma}

\begin{proof}
We prove the contrapositive implication. The argument is local. Assume that $\dim\cF=m>0$. In a small neighborhood we can pick $2d$ vector fields $Y_1^s, Y_2^s,\ldots Y_d^s, Y_1^u, \ldots Y_d^u$ which are $C^1$ regular such that
$$
E^s=span \{Y_1^s, Y_2^s,\ldots Y_d^s\},\,\,\, E^u=span \{Y_1^u, Y_2^u,\ldots Y_d^u\}
$$
and 
$$
T\cF=span \{Y_1^s, Y_2^s,\ldots Y_m^s\},
$$

We will repeatedly use two basic facts about the Lie bracket. First, the bracket is, in fact, a first order differential operator and, hence, is defined for $C^1$ vector fields. The second one is this: if two vector fields are tangent to a foliation then their bracket is also tangent to this foliation (easy direction of the Frobenius theorem).

Because $E^s$ is integrable we have $[Y_i^s, Y_j^s]\in E^s\subset \ker\alpha$. Hence
$$
d\alpha(Y_i^s, Y_j^s)=Y_i^s\alpha(Y_j^s)-Y_j^s\alpha(Y_i^s)-\alpha([Y_i^s, Y_j^s])=0\footnote{Alternatively one can use invariance of $d\alpha$ and $E^s$, and the fact that vectors in $E^s$ contract to arrive at the same conclusion without explicitly using the integrability property.}
$$
Similarly $d\alpha(Y_i^u, Y_j^u)=0$. And by the same token, because $\cF$ integrates jointly with $W^u$ we have $[Y_i^s, Y_j^u]\in T\cF\oplus E^u \subset \ker\alpha$ when $i\le m$ and, hence $d\alpha(Y_i^s, Y_j^u)=0$ when $i\le m$.

We can now calculate $\alpha\wedge (d\alpha)^d(X, Y_1^s, Y_2^s,\ldots Y_d^s, Y_1^u, \ldots Y_d^u)$ using the permutation formula for the wedge product. Recall that if $\omega$ is a $k$-form and $\eta$ is an $l$-form then
$$
(\omega\wedge\eta)(Z_1, Z_2,\ldots Z_{k+l})=\sum_{\sigma\in S_{k+l}}sign(\sigma) \omega(Z_{\sigma(1)}, \ldots, Z_{\sigma(k)})\eta(Z_{\sigma(k+1)},\ldots , Z_{\sigma(k+l)})
$$ 
First, applying this formula for $\omega=\alpha$ and $\eta=(d\alpha)^d$ and using the fact that $Y^{s/u}_i\in\ker\alpha$ we have
\begin{multline*}
\alpha\wedge (d\alpha)^d(X, Y_1^s,\ldots ,Y_d^s, Y_1^u, \ldots ,Y_d^u)=\alpha(X) (d\alpha)^d(Y_1^s,\ldots ,Y_d^s, Y_1^u, \ldots ,Y_d^u)\\
=(d\alpha)^d(Y_1^s,\ldots ,Y_d^s, Y_1^u, \ldots ,Y_d^u)
\end{multline*}
Then to calculate this value we can inductively apply the wedge product formula until we express $(d\alpha)^d(Y_1^s,\ldots ,Y_d^s, Y_1^u, \ldots ,Y_d^u)$ as the sum over all permutations of $d$-fold products of values of $d\alpha$. Note that by the above observations many of these values vanish. Indeed, the only non-vanishing values have the form $d\alpha(Y_i^s, Y_j^u)$ for $i> m$. Since for each permutation the corresponding product can have at most $d-m$ such non-vanishing factors, it has at least $m$ zero factors and, hence, we obtain that $\alpha\wedge (d\alpha)^d(X, Y_1^s,\ldots ,Y_d^s, Y_1^u, \ldots ,Y_d^u)=0$, contradicting the contact property of $\alpha$.
\end{proof}

\subsection{Matching functions and the Subbundle Theorem}

We first recall the matching function technique which we have first introduced in~\cite{GRH} and further developed in~\cite{GRH3, GRH4}. Then we explain the statement of the Subbundle Theorem which was proved in~\cite{GRH3, GRH4}.

Let $\phi_i^t\colon M_i\to M_i$, $i=1,2$ Anosov flows with $C^{r}$ weak stable and unstable foliations, $r\ge 1$. Assume that they are conjugate, $h\circ \phi_1^t=\phi_2^t\circ h$. 
We proceed to explain a certain construction of sub-bundles $E_i$ of the unstable bundles $E_i^u$ via locally matching functions on the local unstable leaves. (Of course, the same construction can be applied on local stable leaves yielding sub-bundles of the stable bundle.)

Recall that the conjugacy $h$ maps leaves of $W_1^u$ to leaves of $W_2^u$. For each $x\in M_1$ consider pairs of $C^{r}$, $r\ge 1$, functions $(\rho^1,\rho^2)$ where $\rho^1$ is defined on an open neighborhood of $x$ in $W^u_1(x)$, $\rho^2$ is defined on an open neighborhood of $h(x)$ in $W^u_2(h(x))$ and such that
$$
\rho^1=\rho^2\circ h.
$$
This relation is what we call a {\it matching relation.} We collect all such pairs of functions into a space $V_x^{r}$
$$
V_x^{r}=\{(\rho^1,\rho^2): \rho^1=\rho^2\circ h\}.
$$
The domains of definition of $\rho^1$ and $\rho^2$ can be arbitrarily small open sets. Also denote by $V_{x,1}^{r}$ the collection of all possible $\rho^1$, that is, projection of $V_x^{r}$ on the first coordinate, and by $V_{x,2}^{r}$ the projection on the second coordinate.

Now we can define linear subspaces $E_i(x)\subset E^u_i(x)$ by intersecting the kernels of all $D\rho^i$ at $x$, $i=1,2$. Namely,
$$
E_i(x)=\bigcap_{ \rho^i\in V_{x,i}^{r}} \ker D\rho^i(x).
$$
We note that subbundles $E_i$ also depend on $r$, which can be taken to be any number $\ge 1$. However in this paper we will only use for specific $r$ given by regularity of invariant distributions.

It turns out that all subspaces $E_i(x)$, $x\in M_i$, $i=1,2$, have the same dimension and give an integrable sub-bundle with certain pleasant properties. Namely, we have the following Subbundle Theorem which was established in~\cite{GRH4}, (and before that for Anosov diffeomorphisms~\cite[Theorem 4.1]{GRH3}).

\begin{theorem}[Subbundle Theorem]
\label{thm_tech}
Let $\phi_i^t\colon M_i\to M_i$, $i=1,2$, be conjugate Anosov flows, $h\circ \phi_1^t=\phi_2^t\circ h$. Assume that both flows have $C^{r}$ stable foliations.
Then there exist $C^{r}$ regular, $D\phi_i^t$-invariant distributions $E_i\subset E_i^u$, such that
\begin{enumerate}
\item the distributions $E_i$ integrate to $\phi_i^t$-invariant foliations $\cF_i\subset W_i^u$;
\item the distributions $E^s_i\oplus E_i$ integrate to an $\phi_i^t$-invariant $C^{r}$ foliation 
which is sub-foliated by both $W^s_i$ and $\cF_i$;
\item the conjugacy $h$ maps $\cF_1$ to $\cF_2$;
\item the restrictions of $h$ to the unstable leaves are uniformly $C^{r}$ transversely to
$\cF_1$;
\item if $(\rho^1,\rho^2)\in V_x^{r}$ is a matching pair then $\rho^i$ is constant on connected local leaves of $\cF_i$.
\end{enumerate}
\end{theorem}

\begin{remark} In~\cite{GRH4} we have defined the matching functions on local weak unstable manifolds instead of local unstable manifolds. We observe that any matching pair on local weak unstable manifolds can be restricted to local unstable manifolds and any matching pair on local unstable manifolds can be pulled back to a matching pair on weak unstable manifolds using local projection along the flow. Also, as explained right after the statement of~\cite[Theorem~2.1]{GRH4}, the subbundles defined through the matching functions on weak unstable manifolds are contained in unstable bundles. Hence, both of these definitions yield the same subbundles $E_i$.
\end{remark}


\subsection{Non-stationary linearization for expanding foliations} Let $\phi^t\colon M\to M$ be a smooth flow which leaves invariant a continuous foliation $W^u$ with uniformly smooth leaves. Assume that $W^u$ is an {\it expanding foliation}, that is, for a sufficiently large $t$
$\|D\phi^t(v)\|>\|v\|$, for all non-zero $v\in E^u$, where $E^u=TW^u$ is the distribution tangent to $W^u$. The following proposition on non-stationary linearization is a special case of the normal form theory developed by Guysinsky and Katok~\cite{GK} and further refined by Kalinin and Sadovskaya~\cite{KS2, K}. We will denote by $D^u$ the restriction of the differential to $E^u$.
\begin{proposition}
\label{prop_normal_forms}
	Let $r\in(1,2]$ and let $\phi^t$, $W^u$ and $E^u$ be as above. Assume that there exist a sufficiently large $t$ such that
	$$
 \|D^u\phi^t(x)\|< m(D^u\phi^t(x))^{r}, x\in M.
 $$
	Then for all $x\in M$ there exists $\cH_x:E^u(x)\to W^u(x)$ such that
	\begin{enumerate}
		\item $\cH_x$ is a smooth diffeomorphism for all $x\in M$;
		\item $\cH_x(0)=x$;
		\item $D_0\cH_x=id$;
		\item $\cH_{\phi^tx}\circ D_x\phi^t=\phi^t\circ \cH_x$ for all $t$;
		\item $D\cH_x$ has $(r-1)$-H\"older dependence along $W^u$;
		\item\label{aff} if $y\in W^u(x)$ then $\cH_y^{-1}\circ \cH_x:E^u(x)\to E^u(y)$ is affine.
 	\end{enumerate}
\end{proposition}

Such family $\{\cH_x, x\in M\}$ is called {\it non-stationary linearization} (also called {\it normal form} or {\it affine structure}) along $W^u$.

It is well-known that non-stationary linearization is unique in appropriate class of linearization. We had difficulty finding a reference for the uniqueness statement which we need. Hence, we provide a precise uniqueness addendum with a proof. We formulate a somewhat more general point-wise uniqueness statement than what we need for the sake of optimality and ease of future reference.

Given a point $x$, let $\kappa_x$ to be the infimum of all $\nu$ so that 
$$
\liminf_{t\to-\infty}\frac{\|D^u_x\phi_t\|^{1+\nu}}{m(D^u_x\phi_t)}=0
$$
Note that the conformal pinching assumption of Proposition~\ref{prop_normal_forms} implies that $\kappa_x\le r-1$.

\begin{add}
\label{add_normal_forms}
Given $x\in M$ assume that $\bar \cH_{\phi^tx}:E^u(\phi^tx)\to W^u(\phi^tx)$, $t\leq 0$, is a family of $C^{1}$ diffeomorphisms satisfying items 2,3,4 (where 4 will only be used for $t\leq 0$) and assume that there exists $\kappa>\kappa_x$ such that $$\sup_{t\leq 0,|z|\leq 1}\frac{\|D_z\bar \cH_{\phi^tx}-Id\|}{|z|^\kappa}<\infty,$$ then $\bar \cH_x=\cH_x$.
\end{add}
\begin{remark}
In the case $\kappa_x$ is a minimum instead of an infimum, \ie the infimum is achieved, we can also take $\kappa=\kappa_x$ in the above addendum. 
\end{remark}

\begin{proof}
	Let $H'_t=\cH_{\phi^tx}^{-1}\circ \bar \cH_{\phi^tx}:E^u(\phi^tx)\to E^u(\phi^tx)$ and observe that $H'_t$ is $C^1$, $H'_t(0)=0$, $D_0H'_t=Id$. 
	Using the main assumption of the addendum, uniform regularity of $\cH_x^{-1}$ and applying the triangle inequality we can easily verify
	$$\sup_{t\leq 0,|z|\leq 1}\frac{\|D_z H'_{t}-Id\|}{|z|^\kappa}<\infty.$$ 
	Also we have the following relation
	\begin{eqnarray*}\label{conjpoint}
		H'_0=(D^u_x\phi^t)^{-1}\circ H'_t\circ D^u_x\phi^t,
	\end{eqnarray*} 
 which is easy to differentiate since two maps are linear and we obtain
 $$D_zH_0'=(D^u_x\phi^t)^{-1}\circ D_{D^u_x\phi^t(z)}H'_t\circ D^u_x\phi^t.
 $$
 Hence 
 
 \begin{eqnarray*}
		\|D_zH_0'-Id\|&=&\|(D^u_x\phi^t)^{-1}\circ D_{D_x\phi^t(z)}H'_t\circ D^u_x\phi^t-Id\|\\
		&\leq&\|(D^u_x\phi^t)^{-1}\|\|D_{D_x\phi^t(z)}H'_t-Id\|\|D^u_x\phi^t\|
	\end{eqnarray*}
	For $t<0$, $D^u_x\phi^t$ is a contraction, consider $t$ so that $|D^u_x\phi^t(z)|\leq 1$, then we get that 
	\begin{eqnarray*}
		\|D_zH_0'-Id\|&\leq&\|(D^u_x\phi^t)^{-1}\|\|D_{D^u_x\phi^t(z)}H'_t-Id\|\|D^u_x\phi^t\|\\
		& \leq&C\|(D^u_x\phi^t)^{-1}\|\|D_x\phi^t(z)\|^{\kappa}\|D_x\phi^t\|\\
		&\leq&\|(D^u_x\phi^t)^{-1}\|\|D^u_x\phi^t\|^{1+\kappa}\||z|^{\kappa}
	\end{eqnarray*}
and the latter goes to $0$ when taking a $\liminf_{t\to-\infty}$, according to the definition of $\kappa$. So $D_zH_0'=Id$ for every $z$, $|z|\le 1$, and hence since $H_0'(0)=0$ we get that $H_0'=Id$, which means $\cH_x=\bar\cH_x$.
\end{proof}

\section{Proof of Theorem~\ref{thm_flows}}

Recall that the weak distributions are $C^r$ by the assumption. Since the strong distributions are given by intersecting with the kernel of the contact form, they are also $C^r$ regular. Hence we apply the Subbundle Theorem~\ref{thm_tech} to $\phi_1^t$ and $\phi_2^t$ and obtain $C^r$ distributions $E_i\subset E^u_i$ and corresponding integral foliations $\cF_i\subset W^u_i$. By item 2 of the Subbundle Theorem we have that $W_i^s$ and $\cF_i$ are jointly integrable. Hence, by Lemma~\ref{lemma_contact} we have $\dim\cF_i=0$, that is, $\cF_i$ are foliations by points. Then item~4 gives uniform $C^r$ smoothness of $h$ along the unstable foliation. 

Entirely symmetric argument yields $C^r$ smoothness of $h$ along the stable foliation. Applying the Journ\'e Lemma first for the  unstable and flow foliations we have that $h$ is $C^{r_*}$ along the weak unstable foliation, then applying Journ\'e Lemma~\cite{J} to weak unstable and stable foliations we obtain that $h$ is $C^{r_*}$. Reversing the roles of the flows we obtain in the same way that $h^{-1}$ is $C^{r_*}$. Hence, $h$ is a $C^{r_*}$ diffeomorphism.

\section{Proof of the corollaries}

\begin{proof}[Proof of Corollary~\ref{thm_flows2}]
The conformal $r$-pinching assumption of the Corollary enables us to apply Proposition~\ref{prop_normal_forms} to both $\phi_1^t$ and $\phi_2^t$. In this way we have normal forms $\cH_x^i$, $i=1,2$, for $\phi_i^t$ along the unstable foliation $W_i^u$. 

By Theorem~\ref{thm_flows} the conjugacy is $C^r$, $r>1$. Define
$$
\bar\cH_x^{1}=\left(h|_{W_1^u(x)}\right)^{-1}\circ\cH_x^2\circ Dh|_{E_1^u(x)}\colon E_1^u(x)\to W^u_1(x)
$$
It is routine to verify that $\bar \cH_x^{1}$ satisfies properties 2-4 of Proposition~\ref{prop_normal_forms}. Also, since $h$ is $C^r$ we have that $D\bar\cH_x^{1}$ is uniformly $C^{r-1}$ at $x$ and, hence, verifies the main assumption of the Addendum~\ref{add_normal_forms}. We invoke the Addendum~\ref{add_normal_forms} and conclude that $\bar\cH_x^{1}=\cH_x^1$, $x\in M_1$. Hence
$$
h|_{W_1^u(x)}=\cH_x^2\circ Dh|_{E_1^u(x)}\circ (\cH_x^1)^{-1},
$$
which is $C^\infty$ regular as the normal forms are smooth for each $x$. Applying the same argument to the stable foliation and then using the Journ\'e Lemma in the same way as in the proof of Theorem~\ref{thm_flows} we establish that $h$ is a $C^\infty$ diffeomorphism.
\end{proof}

\begin{proof}[Proof of Corollaries~\ref{cor2b} and~\ref{cor2}] We will verify that the geodesic flows are bunched with $r=\sqrt 2$ and conformally $\sqrt 2$-pinched. Then applying Corollary~\ref{thm_flows2} finishes the proof of Corollary~\ref{cor2b}. Also notice that both bunching and conformal pinching conditions are open in $C^1$ topology, hence, Corollary~\ref{cor2} also follows.

So let $\phi^t$ be a $\frac{1}{a^2}$-pinched geodesic flow with $a=\sqrt 2$. We can rescale the metric so that all sectional curvatures lie the interval $(-a^2,-1]$. Then we can use the description of stable  (unstable) subbundle as the space of bounded in the future (past) Jacobi fields (see, \eg~\cite[Chapter VI]{E}) and, by comparison with constant-coefficients Jacobi equations $J''-J=0$ and $J''-a^2J=0$ we have
$$
e^t\le \|D\phi^t|_{E^u}\|<e^{at}, t>0,\,\,\,\,\mbox{and}\,\,\,\,\, e^{-at}<\|D\phi^t|_{E^s}\|\le e^{-t}, t>0.
$$
These give bounds on all needed norms and conorms. One can then easily see that bunching with parameter $r$ is implied by $e^{-t}e^{rat}\le e^t$ and conformal $r$-pinching is implied by $e^{at}\le e^{rt}$, which are equivalent to $ra\le 2$ and $a\le r$. Taking $r=a=\sqrt 2$ finishes the proof.
\end{proof}

\begin{proof}[Proof of Corollaries~\ref{cor3} and~\ref{cor4}]
Derivations of these corollaries follow closely~\cite{Ham}. It is well known that negatively curved homotopy equivalent manifolds have orbit equivalent geodesic flows. Then, by the classical application of the Livshits Theorem~\cite[Theorem 19.2.9]{KH}, same marked length spectrum implies existence of a $C^0$ conjugacy $h$ of the geodesic flows. Thus, because we have assumed that stable and unstable foliations are $C^1$ we can apply Theorem~\ref{thm_flows}. Formally speaking, it only yields Lipschitz regularity of $h$. However, in fact, it is easy to overcome the loss of regularity in this case and show that $h$ is a $C^1$ diffeomorphism. Indeed, recall that the loss from $r$ to $r_*$ happens at the very end of the proof of Theorem~\ref{thm_flows} which is due to application of Journ\'e Lemma. However this problem only occurs for integer $r\ge 2$. It is an easy calculus exercise to check that if $h$ is $C^1$ along a pair of transverse foliations then $h$ is a $C^1$ diffeomorphism. 

Now denote by $\alpha_i$ the canonical contact form for $\phi_i^t$. That is, $\alpha_i$ is a $D\phi_i^t$ invariant contact form such that $\alpha_i(X_i)=1$, $X_i=\frac{\partial\phi_i^t}{\partial t}\Big|_{t=0}$, $i=1,2$. Since $h$ is $C^1$ the pull-back form $h^*\alpha_2$ is well-defined and we have
$$
h^*\alpha_2(X_1)=\alpha_2(Dh(X_1))=\alpha_2(X_2)=1
$$
Also
$$
\ker h^*\alpha_2=Dh^{-1}(\ker\alpha_2)=Dh^{-1}(E_2^s)\oplus Dh^{-1}(E_2^u)=E_1^s\oplus E_1^u=\ker\alpha_1
$$
But the value on $X_1$ and the kernel determine a 1-form uniquely. Hence $h^*\alpha_2=\alpha_1$. We have the same for volume forms $\omega_i=\alpha_i\wedge(d\alpha_i)^d$:
$$
h^*\omega_2=h^*\alpha_2\wedge h^*(d\alpha_2)^d=\alpha_1\wedge(d\alpha_1)^d=\omega_1
$$
which implies that total volumes are the same
$$
vol(M)=\int_M\omega_1=\int_Mh^*\omega_2=\int_N\omega_1=vol(N)
$$
finishing the proof of Corollary~\ref{cor3}.

For the last corollary, notice that since $M$ is hyperbolic and $N$ is $1/4$-pinched they have $C^1$ Anosov splitting and, hence, the above proof applies to conclude that they have the same volume.  Since geodesic flows are conjugate they also have the same topological entropy which is well-known to coincide with the volume entropy on the universal covers $\tilde M$ and $\tilde N$. Hence, by the main result of~\cite{BCG} we can conclude that $M$ and $N$ are isometric.
\end{proof} 


 

\end{document}